\documentclass[12pt]{amsart}

\usepackage{amssymb}
\usepackage{pb-diagram}
\usepackage{enumitem}

\newtheorem{theorem}{Theorem}[section]
\newtheorem{proposition}[theorem]{Proposition}
\newtheorem{lemma}[theorem]{Lemma}

\newtheorem{definition}[theorem]{Definition}
\newtheorem{remark}[theorem]{Remark}
\newtheorem{example}[theorem]{Example}

\numberwithin{equation}{section}

\begin{document}

\baselineskip=16pt

\title[Segre invariant, a stratification of the moduli space of coherent systems]{Segre invariant and a stratification of the moduli space of coherent systems}

\author{L. Roa-Leguizam\'on}

\address{Centro de Investigaci\'on en Matem\'aticas, A.C. (CIMAT), Apdo. Postal 402,C.P. 36240, Guanajuato, Gto, M\'exico.}

\email{leonardo.roa@cimat.mx}









\date{\today}

\maketitle

\begin{abstract}
The aim of this paper is to generalize the $m-$Segre invariant for vector bundles to coherent systems.  Let $X$ be a non-singular irreducible complex projective curve of genus $g$ over $\mathbb{C}$ and $(E,V)$ be a coherent system on $X$ of type $(n,d,k)$. For any pair of integers $m, t$, $0 < m < n$, $0 \leq t \leq k$ we define the $(m,t)-$Segre invariant, denoted by $S^\alpha_{m,t}$ and show that $S^\alpha_{m,t}$ induces a semicontinuous function on the families of coherent systems.  Thus, $S^\alpha_{m,t}$ gives a stratification of the moduli space $G(\alpha;n,d,k)$ of $\alpha-$stable coherent systems of type $(n,d,k)$ on $X$ into locally closed subvarieties $G(\alpha;n,d,k;m,t;s)$ according to the value  $s$ of $S^\alpha_{m,t}$.   We study the stratification, determine conditions under which the different strata are non-empty and compute their dimension.
\end{abstract}

\section{Introduction}\label{intro}

Let $X$ be a non-singular irreducible complex projective curve of genus $g$ over $\mathbb{C}$.  The $m-$Segre invariant for vector bundles was first introduced by Lange and Narasimhan in \cite{Lange-Narasimhan} for vector bundles of rank  $n=2$. Then, it was generalized by Brambila-Paz and Lange in \cite{Brambila-Lange} for vector bundles of rank $n \geq 2$ (see also \cite{Russo-Teixidor}).  The $m-$Segre invariant was used to give a stratification of the moduli space $M(n,d)$ of stable vector bundles of rank $n$ and degree $d$ in order to determine topological and geometric properties of $M(n,d)$ (see for instance \cite{Popa}, \cite{Popa-Mike}, \cite{Lange-Newstead1}).  Moreover, it has been generalized to study other moduli spaces (see  \cite{Bhosle-Biswas}, \cite{Choe-Hitching1}).

The aim of this  paper is to generalize  the $m-$Segre invariant for vector bundles to coherent systems.  A coherent  system $(E,V)$ on X consists of a holomorphic vector bundle $E$ on $X$ and a  subspace $V$ of the space of sections $H^0(X,E)$.  Associated with the coherent systems, there is a notion of stability which depends on a real parameter $\alpha$. This notion allows the construction of the moduli space $G(\alpha;n,d,k)$ of $\alpha-$stable coherent systems of  type $(n,d,k)$ and thus leads to a family of moduli spaces.  For a further treatment of the subject see \cite{Newstead}, \cite{Bradlow-Garcia-Prada-Munoz-Newstead1} and \cite{Bradlow}.

Let $\alpha >0$ and $(E,V)$ be a coherent system of type $(n,d,k)$. For  any pair of integers $m, t$, $0 < m < n$, $0 \leq t \leq k$ we define the $(m,t)-$Segre invariant (see Definition \ref{DefiSegreInvariant}) as,
\[S^\alpha_{m,t}(E,V) := (mn)\min \{\mu_\alpha(E,V) - \mu_\alpha(F,W)\}  \]
where the minimum is taken over all principal subsystems (see Remark \ref{basic}) $(F,W)$ of type $(m,d_F,t)$  of $(E,V)$, and $\mu_\alpha$ denote the $\alpha-$slope of coherent systems.   

Using similar techniques as Maruyama in  \cite{Maruyama}, \cite{Maruyama1} and \cite{Maruyama2} we show that the $(m,t)-$Segre invariant  induces a function (called the $(m,t)-$Segre function) on the families of coherent systems and prove 
our first result (see Theorem \ref{Theosemicontinuous}),

\begin{theorem} \label{1theorem}
The $(m,t)-$Segre function is  lower semicontinuous.
\end{theorem}

As consequence of Theorem \ref{1theorem}, the $(m,t)-$Segre invariant yields a stratification of the moduli space $G(\alpha;n,d,k)$ into locally closed subvarieties which we denote as
\[G(\alpha;n,d,k;m,t;s) := \{(E,V) \in G(\alpha;n,d,k) : S^\alpha_{m,t}(E,V) = s\}\]
according to the value  $s$ of $S^\alpha_{m,t}$.  We show a bound for the possible values that can take $s$ (see Proposition \ref{acotainvariante}), determine certain values of $m$, $t$ and $s$ under which  the stratum $G(\alpha;n,d,k;m,t;s)$ is non-empty  (see Theorem \ref{Theoremnonempty}) and compute a bound of its dimension,

\begin{theorem}
Let $\alpha= \frac{p}{q} \in \mathbb{Q}^+$ and $n \geq 2$, $d > 0$, $k \geq 1$ be integer numbers.   Suppose that there exist $n_1, n_2, d_1, d_2 >0$, $t_1, t_2  \geq 0$ integer numbers such that
\begin{equation*} 
n_1n_2(g-1) -d_1n_2 + d_2n_1 + t_2n_1(g-1) - t_1t_2 > 0,
\end{equation*}
\begin{equation*} 
q(n_1d_2-n_2d_1)+p(n_1t_2-n_2t_1) = 1.
\end{equation*}
and  the moduli spaces $G(\alpha;n_1,d_1,t_1)$ and $G(\alpha;n_2,d_2,t_2)$ are non-empty.  Then, the stratum 
\[G(\alpha;n,d,k;n_1,t_1;1/q)\]
is non-empty, where $n= n_1+n_2$, $d= d_1+d_2$ and $k= t_1+t_2$.
\end{theorem}

The paper is organized as follows, Section 2 contains a brief summary of  the $m-$Segre invariant for vector bundles and relevant material on coherent systems.  In Section 3 we define the $(m,t)-$Segre invariant for coherent systems and  coherent systems of subtype $(a)$,   show some technical results which allows to prove  Theorem \ref{1theorem}.  In Section 4 we study the stratification of $G(\alpha;n,d,k)$, determine conditions under which the different strata are  non-empty and compute a bound of their dimension. In section 5 we consider coherent systems of type $(2,13,4)$ on a general curve of genus $6$.

Notation: For a vector bundle $E$ we shall denote by $rk_E$ the rank and by $d_E$ the degree.  We will denote by $\omega_X$ the canonical sheaf on $X$. 

\textbf{Acknowledgments:} The author gratefully acknowledges the many helpful suggestions of  Leticia Brambila-Paz and Herbert Lange during previous versions of the paper and for many stimulating conversations.  The author acknowledges the support of CONACYT Proj CB-251938.
\section{Review of Segre Invariant and Coherent Systems}

Let $X$ be a non-singular irreducible complex projective curve of genus $g$ over $\mathbb{C}$.  This section contains a brief summary  on the $m-$Segre invariant for vector bundles, for more details see \cite{Lange-Narasimhan} and \cite{Brambila-Lange}. We recall the main results that we will use on coherent systems for a further treatment of the subject see \cite{Bradlow} and \cite{Bradlow-Garcia-Prada-Munoz-Newstead1}.   

\subsection{The $m-$Segre invariant for vector bundles}

Let $E$ be a vector bundle  of rank $n$ and degree $d$  on $X$. For any integer $m$, $1 \leq m \leq n-1$ the \textit{$m-$Segre invariant} is defined as
\[S_m(E) := m d - n \max \{d_F\}\]
where the maximum is taken over all subbundles $F$ of rank $m$ of $E$.   The $m-$Segre invariant was first introduced  by Lange and Narasimhan in \cite{Lange-Narasimhan} for vector bundles of rank  $n=2$, then it was generalized by Brambila-Paz and Lange in \cite{Brambila-Lange} and Russo and Teixidor in \cite{Russo-Teixidor}  for vector bundles of rank $n \geq 2$.  Recall that the $slope$ of a vector bundle $E$ denoted by $\mu(E)$, is the quotient
\[\mu(E) := \frac{d}{n}.\]
So, the $m-$Segre invariant can be written as
\[S_m(E) = (nm) \min_{rank \,\, F = m} \{\mu(E)-\mu(F)\}.\]

Note that for a suitable vector bundle $E$, $S_k(E)$ may take arbitrarily negative values (for instance E to be a suitable direct sum of line bundles). However, Hirschowitz in \cite{Hirschowitz} gives an upper  bound
\[S_k(E) \leq m(n-m)(g-1)+(n-1).\]

The following theorem studies the behavior of $S_m$ over families of vector bundles.

\begin{theorem} \emph{\cite[Lemma 1.2]{LangeH}} \label{semicontinuousvectorbundles}
 Let $Y$ be a variety and $\mathcal{E}$ be a family of vector bundles of rank $n$ and degree $d$ parametrized by $Y$. For any integer $m$, $1 \leq m \leq n-1$ the $m-$Segre invariant defines functions 
 \begin{align*}
 S_m: Y & \longrightarrow \mathbb{Z} \\
 y & \longmapsto S_m(\mathcal{E}_y).
 \end{align*}
  The function $S_m$ is lower semicontinuous.
 \end{theorem}

Denote by $M(n,d)$ the moduli space of stable vector bundles of rank $n$ and degree $d$ on $X$.  By  Theorem \ref{semicontinuousvectorbundles} the function $S_m:M(n,d) \longrightarrow \mathbb{Z}$ gives a stratification of $M(n,d)$ into locally closed subvarieties 
\[M(n,d;m;s) : = \{E \in M(n,d) : S_m(E) = s\}\]
according to the value s of $S_m$.  

The stratification of $M(2,d)$  was studied in \cite{Lange-Narasimhan}.	There was shown that for $s > 0$, $s \equiv d$ $\mod 2$ the algebraic variety $M(2,d;s)$ is non-empty, irreducible and of dimension
\[dim \,\ M(2,d;s) = \begin{cases}
3g + s -2, \,\, & \text{if $s \leq g-2$} \\
4g-3, \,\, & \text{if $s \geq g-1$.}
\end{cases}\]

For the stratification of $M(n,d)$, it is shown in \cite{Brambila-Lange} that for $g \geq \frac{n+1}{2}$ and $0 < s \leq m(n-m)(g-1)+(n+1)$, $s \equiv md$ $\mod n$ the variety $M(n,d;m;s)$ is non-empty and it has an irreducible component $M^0(n,d;m;s)$ of dimension 
\[dim\,\, M^0(n,d;m;s) = \begin{cases}
(n^2+m^2-nm)(g-1)+s-1, \,\, & \text{if $s \leq m(n-m)(g-1)$} \\
n^2(g-1)+1, \,\,  & \text{if $s \geq m(n-m)(g-1)$.}
\end{cases}\]
The main difficulty in the study of the stratification of $M(n,d)$ is to show that the different varieties $M(n,d;m;s)$ are non-empty.  This was shown in \cite{LangeH1} and \cite{Hirschowitz}  for $s \geq m(n-m)(g-1)$ for the generic case and for $s \leq \frac{m(n-m)(g-1)}{\max \{m, n-m\}}$ in \cite{Ballico-Brambila-Russo}. Special cases were considered by  Russo and Teixidor in \cite{Russo-Teixidor}.

Such stratification has been used by many authors to get topological and geometric properties of $M(n,d)$ and the $m-$Segre invariant has been generalized to study others moduli spaces.  For instance, Popa in \cite{Popa} determined a bound for the $m-$Segre invariant, this result is applied to the study of generalized theta line bundles on the moduli space $M(n,d)$. Bhosle and Biswas in \cite{Bhosle-Biswas} defined an analogue of the $m-$Segre invariant for parabolic bundles  in order to study the moduli space $\mathcal{M}(n,d)$ of stable parabolic bundles of rank $n$ and degree $d$ with fixed parabolic structure at a finite set of distinct closed points of $X$. Similarly, Choe and Insong in \cite{Choe-Hitching1} studied the moduli space $M_{2n}$ of semistable symplectic bundles of rank $2n$ on $X$.

\subsection{Coherent systems}

A \textit{coherent system} or Brill-Noether pair on $X$ of type $(n,d,k)$ is  a pair $(E,V)$ where $E$ is a holomorphic vector bundle on $X$ of rank $n$ and degree $d$  and $V \subseteq H^0(X,E)$ is a subspace of dimension $k$.  
 
\begin{remark} \label{basic}\emph{
\begin{itemize}
\item A coherent subsystem of $(E,V)$ is a coherent system $(F,W)$ such that $F$ is a subbundle of $E$ and $W \subseteq V \cap H^0(X,F)$. The subsystem  $(F,W)$ is called \textit{principal} if $W = V \cap H^0(X,F)$. 
\item A quotient coherent system of $(E,V)$ is a coherent system $(G,Z)$ together with a homomorphism $\phi:(E,V) \longrightarrow (G,Z)$ such that $E \longmapsto G$ and $V \longmapsto Z$ are surjective.
\end{itemize}
   }
 \end{remark}

\begin{remark} \label{subdefineextension}\emph{ 
In general, a subsystem does not define a quotient system. However, any  principal subsystem $(F,W)$ of $(E,V)$ defines a corresponding quotient system $(G,Z)$ which fit in the exact sequence 
\[0 \longrightarrow (F,W) \longrightarrow (E,V) \longrightarrow (G,Z) \longrightarrow 0.\]
}
\end{remark}

\begin{definition} \label{Family1}\emph{
A \textit{family of coherent systems} of type $(n,d,k)$ on $X$ parametrized by a variety $T$ consists of a pair $(\mathcal{E}, \mathcal{V})$ where;
\begin{itemize}
\item $\mathcal{E}$ is a family of vector bundles on $X$ parametrized by $T$ such that $\mathcal{E}_t = \mathcal{E} _{X \times \{t\}}$ has degree $d$ and rank $n$ for all $t \in T$.
\item $\mathcal{V}$ is a locally free subsheaf of $p_{T*}\mathcal{E}$ of rank $k$ such that the fibers $\mathcal{V}_t$ map injectively to $H^0(X, \mathcal{E}_t)$ for all $t \in T$ where  $p_T$ denotes the canonical projection of $X \times T$ on $T$.
\end{itemize}}
\end{definition}

By  Serre-Grothendieck duality Theorem and by \cite[Lemme 4.9.]{LePotier} the vector spaces $Ext^1_{p}(\mathcal{E}, \omega_{X \times T/T})_t$ and $H^0(X,\mathcal{E}_t)$ are duals, hence it implies another way of stating the Definition \ref{Family1}.

\begin{definition}\label{Family2}\emph{
A  \textit{family of coherent systems} of type $(n,d,k)$ on $X$ parametrized by a variety $T$ consists of a pair $(\mathcal{E}, \Gamma)$ where;
\begin{itemize}
\item $\mathcal{E}$ is a family of vector bundles on $X$ parametrized by $T$ such that $\mathcal{E}_t = \mathcal{E} _{X \times \{t\}}$ has degree $d$ and rank $n$ for all $t \in T$.
\item $\Gamma$ is a locally free quotient sheaf of $Ext^1_{p_T}(\mathcal{E}, \omega_{X \times T/T})$ of rank $k$ on $T$ where $p_T$ denotes the canonical projection of $X \times T$ on $T$  and $\omega_{X \times T/T}= p^*_X(\omega_X)$  the relative canonical sheaf by the projection $p_T$.
\end{itemize}}
\end{definition} 

Associated to the coherent systems there is a notion of stability which depends on a real parameter $\alpha$. For a real number $\alpha$, the $\alpha-$slope of a coherent system $(E,V)$ of type $(n,d,k)$ is defined by
\[\mu_\alpha(E,V) := \frac{d}{n} + \alpha \frac{k}{n}.\]
We say, $(E,V)$ is $\alpha-$stable (resp. $\alpha-$semistable) if for all proper subsystems $(F,W)$,  
\[\mu_\alpha(F,W) < \mu_\alpha(E,V), \,\, \text{(resp. $\leq$)}.\]


The moduli space of $\alpha-$stable coherent systems of fixed type was  constructed  by Le Potier in \cite{LePotier}, by King and Newstead in \cite{King-Newstead} and by Ragavendra and Vishwanath in \cite{Ragavendra-Vishwanath} by the methods of Geometric Invariant Theory.    We shall denote the moduli space of $\alpha-$stable coherent systems of type $(n,d,k)$ by $G(\alpha;n,d,k)$. 

Necessary conditions for non-emptiness of $G(\alpha;n,d,k)$ are $d > 0$, $\alpha > 0$, $(n-k)d < \alpha$, for a discussion of recent progress on the non-emptiness problem we refer the reader to \cite{Newstead}. Now, we describe some well known facts about the moduli space of coherent systems.  

\begin{definition}\emph{
We say that $\alpha > 0$ is a virtual critical value if it is numerically possible to have a proper subsystem $(F,W)$ of type $(m,d_F,t)$ such that $\frac{t}{m} \neq \frac{k}{n}$ but $\mu_\alpha(E,V) =\mu_\alpha(F,W)$.  If there is a coherent system $(E,V)$ and a subsystem $(F,W)$ such that this actually holds, we say $\alpha$ is a critical value. We say $\alpha = 0$ is a critical value.}
\end{definition}

For numerical reasons for any $(n,d,k)$ there are finitely many critical values
\[0= \alpha_0 < \alpha_1 < \ldots < \alpha_L < \begin{cases}
\frac{d}{n-k},  \,\, &\text{if $k < n$}\\
\infty,  \,\, &\text{if $k \geq n$.}
\end{cases}\]
These induce a partition of the $\alpha-$range into a set of open intervals  such that within the interval $(\alpha_i, \alpha_{i+1})$ the property of $\alpha-$stability is independent of $\alpha$. If $k \geq n$ the moduli spaces coincide for any two different values of $\alpha$ in the range $(\alpha_L, \infty)$, (see \cite[Proposition 4.6.]{Bradlow-Garcia-Prada-Munoz-Newstead1}).  We will denote by $G_i=G_i(n,d,k)$ the moduli space in the interval $(\alpha_i, \alpha_{i+1})$ and by $G_L := G_L(n,d,k)$ for $\alpha > \alpha_L$ .

Coherent systems form an abelian category and the functors $Hom((E,V),-)$ are left exact.  Hence their derived functors denoted by  $Ext^i((E,V),-)$ are well defined. For a more detailed treatment we refer the reader to \cite{MinHe}.

Given two coherent systems $(E_1,V_1)$, $(E_2,V_2)$ of type $(n_1,d_1,k_1)$, $(n_2,d_2,k_2)$, respectively one defines the groups 
\begin{align*}
\mathbb{H}_{21}^0 &:= Hom((E_2,V_2),(E_1,V_1)), \\
\mathbb{H}_{21}^i &:= Ext^i((E_2,V_2),(E_1,V_1)), \,\, \text{for $i > 0$.}
\end{align*}
and consider the long exact sequence 
\begin{align*}
0 & \longrightarrow Hom((E_2,V_2), (E_1,V_1)) \longrightarrow Hom(E_2,E_1) \longrightarrow Hom(V_2,H^0(X,E_1)/V_1) \\
& \longrightarrow Ext^1((E_2,V_2), (E_1,V_1)) \longrightarrow Ext^1(E_2,E_1) \longrightarrow Hom(V_2,H^1(X,E_1)) \\
& \longrightarrow Ext^2((E_2,V_2), (E_1,V_1)) \longrightarrow 0.
\end{align*}
From \cite[Proposition 3.2.]{Bradlow-Garcia-Prada-Munoz-Newstead1} follows
\[\dim \,\, Ext^1((E_2,V_2),(E_1,V_1) )= C_{21}+ dim \,\, \mathbb{H}_{21}^0 + dim \,\, \mathbb{H}_{21}^2,\]
where
\[C_{21} =n_1n_2(g-1) - d_1n_2 + d_2n_1 + k_2d_1 - k_2n_1(g-1) - k_1k_2.\]

\begin{proposition}
The space of equivalence classes of extensions 
\begin{equation} \label{extensioncohsystems}
0 \longrightarrow (E_1,V_1) \longrightarrow (E,V) \longrightarrow (E_2,V_2) \longrightarrow 0
\end{equation}
is isomorphic to $Ext^1((E_2,V_2), (E_1,V_1))$. Moreover, the quotient of the space of non-trivial extensions by the natural action of $\mathbb{C}^*$ can be identified with the projective space \[\mathbb{P}Ext^1((E_2,V_2), (E_1,V_1)).\]
\end{proposition}

Note that if $Aut(E_i,V_i) = \mathbb{C}^*$, then the isomorphism classes of $(E,V)$ appearing in the middle of (\ref{extensioncohsystems}) will be parametrized by $\mathbb{P}Ext^1((E_2,V_2), (E_1,V_1))$.

\begin{definition}
For any $(n,d,k)$, the Brill-Noether number $\beta(n,d,k)$ is defined by
\[\beta(n,d,k) = n^2(g-1)+1-k(k-d+n(g-1)).\]
\end{definition}

\section{Segre Invariant for Coherent Systems}

In  this section, we introduce the $(m,t)-$Segre invariant for coherent systems and show some properties which will be used later.   We show that the $(m,t)-$Segre invariant induces a semicontinuous function on the families of coherent systems, the proof proceeds as Maruyama in \cite{Maruyama}, \cite{Maruyama1} and \cite{Maruyama2} for the $m-$Segre invariant for vector bundles. Unless otherwise stated we assume that $\alpha > 0$ is a rational number. 

Let $(E,V)$ be a coherent system of type $(n,d,k)$.   For any pair of integers $m, t$, $0 < m < n$, $0 \leq t \leq k$ consider the set 
\[
P_{m,t}(E,V):= \{(F,W) \subset (E,V) :  \,\, \text{$rk_F = m$,  $dim \,\, W =t$  and  $(F,W)$ principal}  \}.
\]

By Remark \ref{subdefineextension}, the set $P_{m,t}(E,V)$ consists of all coherent subsystems $(F,W)$ of type $(m,d_F,t)$  of $(E,V)$ for which there exists  an  exact sequence of coherent systems
\[0 \longrightarrow (F,W) \longrightarrow (E,V) \longrightarrow (G,Z) \longrightarrow 0.\]

\begin{remark}\emph{
Let $(E,V)$ be a coherent system of type $(n,d,k)$.
\begin{itemize}
\item The set $P_{m,0}(E,V)$ is not in correspondence with the set of all subbundles of rank $m$ of $E$. There exists subbundles $F \subset E$ of rank $m$ for which $dim \,\, V \cap H^0(X,F) > 0$.
\item The set $P_{n,t}(E,V)$ is empty for all $t \neq k$ and $P_{n,k}(E,V) = \{(E,V)\}$.
\end{itemize}}
\end{remark}

\begin{remark} \label{slopebounded}\emph{
Since the degrees of subbundles of $E$ are bounded above} $($see \cite[Lemma 5.4.1]{LePotier2}$)$, \emph{we have that for a fixed value of $\alpha$ the coherent system $(E,V)$ does not admit subsystems of $\alpha-$slope high arbitrary, i.e,  the set 
\[\{\mu_\alpha(F,W) : (F,W) \subset (E,V)\}\]
is bounded above.}
\end{remark}

\begin{definition} \label{DefiSegreInvariant}\emph{
Let $\alpha > 0$ and $(E,V)$ be a coherent system of type $(n,d,k)$.  The \textit{$(m,t)-$Segre invariant  for coherent systems} denoted by $S^\alpha_{m,t}$ is defined by
\[
S^\alpha_{m,t}(E,V)   : = 
(mn) \min_{(F,W) \in P_{m,t}(E,V)} \{\mu_\alpha(E,V) - \mu_\alpha(F,W)\}, 
\] 
provided that $P_{m,t}(E,V) \neq \emptyset$.  If $P_{m,t}(E,V) = \emptyset$, we define $S^\alpha_{m,t}(E,V)= \infty$.}
\end{definition}

By Remark \ref{slopebounded}, $S^\alpha_{m,t}(E,V)$ is a rational number well defined depending only on $(E,V)$, $m$ and $t$. Since  any $(F,W) \in P_{m,t}(E,V)$ defines an exact sequence of coherent systems 
\[0 \longrightarrow (F,W) \longrightarrow (E,V) \longrightarrow (G,Z) \longrightarrow 0,\]
Definition \ref{DefiSegreInvariant} is equivalent to
\[S^\alpha_{m,t}(E,V) = n(n-m) \min \{\mu_\alpha(G,Z) - \mu_\alpha(E,V)\}\]
where the minimum is taken over any quotient coherent system of $(E,V)$ of type $(n-m, d-d_F, k-t)$.

Here are some basic properties of this concept.

\begin{proposition} \label{acotainvariante}
Let $\alpha > 0$ and $(E,V)$ be a coherent system of type $(n,d,k)$, then
\begin{enumerate}
\item $(E,V)$ is $\alpha-$stable \emph{(resp. $\alpha-$semistable)}, if and only if \[S^\alpha_{m,t}(E,V) > 0, \,\,\, \text{\emph{(resp. $\geq 0$)}}\] for any pair of integers $m, t$, $0 < m < n$, $0 \leq t \leq k$.
\item If $P_{m,t}(E,V) \neq \emptyset$ then 
\[S^\alpha_{m,t}(E,V) \leq m(n-m)(g-1)+(n-1)+\alpha(mk-nt).\]
\end{enumerate}
\end{proposition}

\begin{proof}
The proof of $(1)$ follows directly of the definition of $\alpha-$stability for coherent systems.  For (2), Hirschowitz in \cite{Hirschowitz} showed that $S_m(E) \leq m(n-m)(g-1)+(n-1)$, then for $(F,W) \in P_{m,t}(E,V)$ we would have $S^\alpha_{m,t}(E,V) \leq m(n-m)(g-1)+(n-1)+\alpha(mk-nt)$ as claimed. 
\end{proof}


A coherent system $(F,W) \in P_{m,t}(E,V)$ is called \textit{maximal}  if and only if 
\[S^\alpha_{m,t}(E,V) = (mn) (\mu_\alpha(E,V)- \mu_\alpha(F,W)).\]

In the following section we define the $(m,t)-$Segre function and prove that it is lower semicontinuous.

\subsection{Semicontinuity of the Segre function}

Let $Y$ be a variety and $(\mathcal{E}, \mathcal{V})$ be a family of coherent systems on $X$ of type $(n,d,k)$ parametrized by $Y$.  For any pair of integers $m$, $t$, $0 < m< n$, $0 \leq t \leq k $  the \textit{$(m,t)-$Segre function} is defined as
\begin{align*}
S^\alpha_{m,t} : Y &\longrightarrow \mathbb{R} \cup \{\infty\} \\
y &\longmapsto S^\alpha_{m,t}(\mathcal{E}, \mathcal{V})_y.
\end{align*}

The aim of this section is to prove the following theorem.

\begin{theorem} \label{Theosemicontinuous}
The $(m,t)-$Segre function is lower semicontinuous.
\end{theorem}

The proof of the theorem consists in showing that for any $b \in \mathbb{R}$ the set 
\begin{equation*}\label{semiopen}
A_b := \{y \in Y :b < (nm)(\mu_\alpha(\mathcal{E}, \mathcal{V})_y - \mu_\alpha(F,W)) \,\, \text{for any $(F,W) \in (\mathcal{E}, \mathcal{V})_y$}\}
\end{equation*}
is open in $Y$.

To show that $A_b$ is open in $Y$ we proceed as Maruyama in \cite{Maruyama}, \cite{Maruyama1} and \cite{Maruyama2} for the $m-$Segre invariant for  vector bundles.  We define coherent systems of subtype $(a)$ and show that it is an open condition.

\subsubsection{Coherent Systems of Subtype $(a)$} 

Let $n \geq 1$, $k \geq 0$ be integer numbers. We will denote by $(a)$ a sequence of $(n-1)(k+1)$ rational numbers, that is
\[(a) := (a_{ij} : 0 < i <n, \,\, 0 \leq t \leq k) \in \mathbb{Q}^{(n-1)(k+1)}.\]
The sequence $(a):= (a_{ij}= 0, \,\, \text{for all pair $i, j$})$ will be denoted by $(0)$. Let $(a), (b) \in \mathbb{Q}^{(n-1)(k+1)}$ we say $(a)$ is greater than or equal than $(b)$  denoted by $(a) \geq (b)$, if $a_{ij} \geq b_{ij}$ for any pair $i,j$.  Hence, the set $\mathbb{Q}^{(n-1)(k+1)}$ is a partial order set.

The following definition extends the usual notion of $\alpha-$stability of coherent systems.

\begin{definition}\emph{
Let $\alpha > 0$, $(E,V)$ be a coherent system of type $(n,d,k)$ and $(a) = (a_{ij}) \in \mathbb{Q}^{(n-1)(k+1)}$.}
\begin{itemize}
\item \emph{$(E,V)$ is of \textit{$\alpha-$subtype} $(a)$ if  
\[\mu_\alpha(F,W) < \mu_\alpha(E,V) + a_{m,t}\]
for all coherent subsystems $(F,W) \in P_{m,t}(E,V)$.}

\item \emph{$(E,V)$ is of \textit{$\alpha-$cotype} $(a)$ if 
\[\mu_\alpha(E,V) - a_{m',t'} < \mu_\alpha(G,Z)\]
for all quotient coherent system $(G,Z)$ of $(E,V)$ of type $(m', d_G, t')$.}
\end{itemize}
\end{definition}

Unless otherwise stated we assume that $\alpha > 0$ is a fixed value.  For simplicity of notation we write subtype $(a)$ (resp. cotype) instead $\alpha-$subtype $(a)$ (resp. $\alpha-$cotype).

Here are some elementary properties of these concepts.

\begin{remark}\label{properties}
\emph{
Let  $(E,V)$ be a coherent system of type $(n,d,k)$ and $(a)$, $(b) \in \mathbb{Q}^{(n-1)(k+1)}$. 
\begin{enumerate}
\item $(E,V)$ is $\alpha-$stable if and only if it is of subtype $(0)$.
\item If $(E,V)$ is of subtype $(a) := (a_{ij})$ with $a_{ij} < 0$ for any $i,j$, then $(E,V)$ is $\alpha-$stable.
\item If $(E,V)$ is of cotype $(a):= (a_{ij})$ with $a_{ij} < 0$ for any $i,j$, then $(E,V)$ is $\alpha-$stable.
\item If $(E,V)$ is of subtype $(a)$, then it is of subtype $(b)$ for all $(a) \leq (b)$.
\item If $(E,V)$ is of cotype $(a)$, then it is of cotype $(b)$ for all $(a) \leq (b)$.
\end{enumerate}}
\end{remark}

The following proposition establishes a relationship between coherent systems of subtype $(a)$ and cotype $(b)$.

\begin{proposition} \label{subtypecotype}
Let $(E,V)$ be a coherent system of type $(n,d,k)$.  The coherent system $(E,V)$ is of subtype $(a)$, if and only if it is of cotype $(b)$ where $(b)$ is the sequence defined by
\[(b) := (b_{ij}= a_{n-i, k-j}\frac{n-i}{i} : 0 < i <n, \,\, 0 \leq j \leq k).\]
\end{proposition}

\begin{proof}
Let $(a):= (a_{i,j})$, $(b) := (b_{ij}= a_{n-i, k-j}\frac{n-i}{i} : 0 < i <n, \,\, 0 \leq j \leq k) \in \mathbb{Q}^{(n-1)(k+1)}$ and $(E,V)$ be a coherent system of type $(n,d,k)$ of subtype $(a)$.  Let $(G,Z)$ be a quotient coherent system of $(E,V)$ of type $(m',d_G,t')$ and $(F,W)  \in P_{m,t}(E,V)$, which fit into the following exact sequence 
\[0 \longrightarrow (F,W) \longrightarrow (E,V) \longrightarrow (G,Z) \longrightarrow 0.\]

Since $(E,V)$ is of subtype $(a)$ the following inequality holds
\[\mu_\alpha(F,W) < \mu_\alpha(E,V)+ a_{m,t}\]
that is
\[n (d_F +  \alpha t ) < m( d+  \alpha k ) + (m n) a_{m,t},\]
which is equivalent to 
\[n [d- d_G +  \alpha (k-t')]  < (n-m') (d +   \alpha k)  + n (n-m')    a_{n-m',k-t'}\]
and this gives 
\[ \mu_\alpha(E,V)- a_{n-m', k-t'}\frac{n-m'}{m'} < \mu_\alpha(G,Z).\]
That is
\[\mu_\alpha(E,V) - b_{m',t'} < \mu_\alpha(G,Z).\]
Therefore, $(E,V)$ is of cotype $(b)$.
In the same way, we can see that if $(E,V)$ is of cotype $(b)$, then it is of subtype $(a)$ which completes the proof.
\end{proof}

Let $(a) \in \mathbb{Q}^{(n-1)(k+1)}$ and $a_{m,t}$ be the $(m,t)-$member of the sequence $(a)$.  For $1 < m< n$, we denoted by $(a-a_{m,t})$ the sequence defined as 
\[(a-a_{m,t}) := (a_{i,j}-a_{m,t}: 0 < i <m, 0 \leq j \leq t) \in \mathbb{Q}^{(n-m)(k-t)}.\]

\begin{example}\emph{
Let $n=3$, $k=2$ and $(a) = (a_{1,0}, a_{1,1}, a_{1,2}, a_{2,0}, a_{2,1}, a_{2,2}) \in \mathbb{Q}^6$. The sequence $(a-a_{21})$ is defined by
\[(a-a_{21}) := (a_{10}-a_{21}, a_{11}-a_{21}) \in \mathbb{Q}^2.\]}
\end{example}

The following lemma yields information about coherent systems that are not of subtype $(a)$.

\begin{lemma} \label{nocotype}
Let $\alpha > 0$ and $(E,V)$ be a coherent system of type $(n,d,k)$.  If $(E,V)$ is not of subtype $(a)$, then there exist a pair of integers $m$, $t$, $0 < m <n$,  $0 \leq t \leq k$ and a coherent system in $P_{m,t}(E,V)$ such that it is of subtype $(a - a_{m,t})$ and it has $\alpha-$slope greater than or equal to $\frac{d + \alpha k}{n} + a_{m,t}$.
\end{lemma}

\begin{proof}
Since $(E,V)$ is not of subtype $(a)$ there exist a pair of integers $m',t'$, $0 < m' < n, 0 \leq t' \leq k$ and a coherent system  $(F',W') \in P_{m',t'}(E,V)$ such that
\begin{equation} \label{inequality}
\mu_\alpha(F',W') \geq \mu_\alpha(E,V) + a_{m',t'}.
\end{equation}
Suppose that $(F',W')$ is not of subtype $(a-a_{m',t'})$.  By induction on rank we obtain a pair of integers $m, t$, $0 < m <m'$, $0 \leq t \leq t'$ and a coherent system $(F,W) \in P_{m,t}(F',W') \subset P_{m,t}(E,V)$ which satisfies 
\begin{equation} \label{resultado}
\mu_\alpha(F,W) \geq \mu_\alpha(F',W')+ a_{m,t}- a_{m',t'}.
\end{equation}
Replacing (\ref{inequality})  in  (\ref{resultado}) follows that
\begin{align*}
\mu_\alpha(F,W) & \geq \mu_\alpha(E,V)+ a_{m',t'} + a_{m,t}-a_{m',t'}  \\
& = \mu_\alpha(E,V) + a_{m,t} \\
& = \frac{d+\alpha k}{n} + a_{m,t}.
\end{align*}
Therefore $(F,W)$ is one of the systems in $P_{m,t}(E,V)$ which satisfies the lemma.
\end{proof}

Let us denote by $A(n,d,k)$  the set of isomorphism classes of coherent systems on $X$ of type $(n,d,k)$.  For fix $\alpha > 0$ and $(a) \in \mathbb{Q}^{(n-1)(k+1)}$ let 
\[B((a)) := \{(E,V) \in A(n,d,k) : \text{$(E,V)$ is of subtype $(a)$}\}.\]  
The following lemma show that the set $B((a))$ is bounded.  The proof proceeds as Le Potier in \cite[Th\'eoreme 4.11.]{LePotier} and use  the following result. 

\begin{proposition} \cite[Proposition 2.6.]{King-Newstead} \label{propKN}
For fixed $n, d, b$, there is a bounded family containing all torsion-free sheaves $E$ on $X$ with $rk_E = n$, $d_E = d$ and such that all non-zero subsheaves $F$ of $E$ have slope $\mu(F) \leq b$.
\end{proposition}

\begin{lemma} \label{lemmabounded}
The set $B((a))$ is bounded.
\end{lemma}

\begin{proof}
Let $(a) \in \mathbb{Q}^{(n-1)(k+1)}$ and $\bar{a} := \max_{a_{i,j} \in (a)} a_{i,j}.$  Taking $b= \frac{d + \alpha k}{n} + \bar{a}$ in Proposition \ref{propKN} it follows that the set of isomorphism classes of vector bundles occurring in $B((a))$ is bounded.  Let $\mathcal{E}$ be the family of vector bundles on $X$ parametrized by $S$ such that for each $(E,V) \in B((a))$, $E$ is isomorphic to $\mathcal{E}_s$ for some $s \in S$.

Let $p_X$ and $p_S$ the canonical projections of $X \times S$ on $X$ and $S$, respectively.  We will denote by  $\omega_{X \times S/S} = p^*_X(\omega_X)$ the relative canonical sheaf by the projection $p_S$.  Consider the sheaf $\underline{Ext}^1_{p_{S}}(\mathcal{E}, \omega_{X \times S/S})$ on $S$. By \cite[Lemme 4.9]{LePotier} the fiber $\underline{Ext}^1_{p_{S}}(\mathcal{E}, \omega_{X \times S/S})_s$ is isomorphic to $Ext^1( \mathcal{E}_s, \omega_X)$ and by change-base Theorem 
 \[Ext^1(\mathcal{E}_s, \omega_X) \cong H^1(X, \mathcal{E}_s^* \otimes \omega_X).\]

Let us denote by  
\[\pi:\mathcal{G} = \mathcal{G}rass(\underline{Ext}^1_{p_{S}}(\mathcal{E}, \omega_{X \times S/S}),k) \longrightarrow S\]
the Grassmannian of quotient coherent locally free sheaves of rank $k$, equipped with the universal  family  
\[ \pi^*(\underline{Ext}^1_{p_{S}}(\mathcal{E}, \omega_{X \times S/S})) \to \varUpsilon \to  0 .\]
By change-base Theorem 
\[\pi^*(\underline{Ext}^1_{p_{S}}(\mathcal{E}, \omega_{X \times S/S})) \cong
\underline{Ext}^1_{p_{S}}(\mathcal{E}', \omega_{X \times \mathcal{G}/\mathcal{G}})\] where $\mathcal{E}'= (\pi \times id_X)^*\mathcal{E}$.

Hence the pair $(\mathcal{E}', \varUpsilon)$ defines a family of coherent systems on $X$ parametrized by $\mathcal{G}$ in the sense of the Definition \ref{Family2}.  Note that by the way in which we build the family $(\mathcal{E}', \varUpsilon)$ we have that every member of $B((a))$ is isomorphic to one of $\{(\mathcal{E}', \varUpsilon)_y : y \in Y \}$ which proves the lemma.
\end{proof}

In the remainder of this section we show some local properties of the coherent systems of subtype $(a)$, without loss of generality we assume that $Y$ is the spectrum of a discrete valuation ring.  Denote by $y$ (resp. $y_0$) the generic point (resp. closed point) of $Y$.

\begin{lemma}\label{lemmauniversalquotient}
Let  $(\mathcal{E}, \mathcal{V})$ be a family of coherent systems on $X$ parametrized by $Y$.    Let $(G,Z)_y$ be a quotient coherent system of $(\mathcal{E}, \mathcal{V})_y$.  Then there exists a unique quotient family of coherent systems $(\mathcal{G}, \mathcal{Z})$ of $(\mathcal{E}, \mathcal{V})$ such that the restriction on $X \times \{y\}$ is $(G,Z)_y$.
\end{lemma}

\begin{proof}
Let  $(\mathcal{E}, \mathcal{V})$ be a family of coherent systems on $X$ parametrized by $Y$ and $(G,Z)_y$ be a quotient coherent system of $(\mathcal{E}, \mathcal{V})_y$.  As $G_y$ is a quotient bundle of $\mathcal{E}_y$, from  \cite[Lemme 3.7.]{Grothendieck} follows that there exists a unique quotient bundle $\mathcal{G}$ of $\mathcal{E}$ on $X \times Y$, flat over $Y$ such that $\mathcal{G}_y$ is precisely $G_y$.  Moreover, as the morphism $ \mathcal{V}_y \longrightarrow Z_y$ is surjective from \cite[Lemme 3.7.]{Grothendieck}  there is a unique quotient bundle $\mathcal{Z}$ of $\mathcal{V}$, flat over $Y$ such that $\mathcal{Z}_y$ is $Z_y$.  Therefore, $(\mathcal{G}, \mathcal{Z})$ is a family of coherent systems on $X$ parametrized by $Y$ that satisfies the lemma.
\end{proof}

The following lemma shows that the properties subtype and cotype are stable under specialization.

\begin{lemma} \label{specialization}
Let $\alpha > 0$  and  $(\mathcal{E}, \mathcal{V})$ be a family of coherent systems of type $(n,d,k)$ parametrized by $Y$.
\begin{enumerate}
\item If $(\mathcal{E}, \mathcal{V})_{y_0}$  is of cotype $(a)$, then $(\mathcal{E}, \mathcal{V})_y$ is  of cotype $(a)$.
\item If $(\mathcal{E}, \mathcal{V})_{y_0}$ is  of subtype $(a)$, then  $(\mathcal{E}, \mathcal{V})_y$ is  of subtype $(a)$.
\end{enumerate}
\end{lemma}

\begin{proof} 
Let $\alpha > 0$  and  $(\mathcal{E}, \mathcal{V})$ be a family of coherent systems of type $(n,d,k)$ parametrized by $Y$.
\begin{enumerate}
\item Suppose $(\mathcal{E}, \mathcal{V})_y$ is not of cotype $(a)$, then there exist a pair of integers $m'$, $t'$, $0 < m' < n$, $0 \leq t' \leq k$ and a quotient coherent system $(G,Z)_y$ of $(\mathcal{E}, \mathcal{V})_y$ of type $(m',d_G,t')$ which satisfies 
\[\mu_\alpha(G,Z)_y \leq \mu_\alpha(\mathcal{E}, \mathcal{V})_y - a_{m',t'}.\]
By Lemma \ref{lemmauniversalquotient} there exists a family of coherent systems $(\mathcal{G},\mathcal{Z})$ on $X$ parametrized by $Y$ such that the restriction on $X \times \{y\}$ is $(G,Z)_y$.  Since the degree, the rank and the dimension are invariants in the family, it follows that
\[\mu_\alpha(G,Z)_{y_0} \leq \mu_\alpha(\mathcal{E}, \mathcal{V})_{y_0}- a_{m',t'}.\]
Hence, $(\mathcal{E}, \mathcal{V})_{y_0}$ is not of cotype $(a)$.

\item Suppose $(\mathcal{E}, \mathcal{V})_y$ is not of subtype $(a)$, by Proposition \ref{subtypecotype} the coherent system $(\mathcal{E}, \mathcal{V})_y$ is not of cotype $(b)$, where 
\[(b) = (b_{i,j} = a_{n-i,k-j}\frac{n-i}{i} : 0 < i < n, \,\, 0 \leq j \leq k ),\]
by (1), we have $(\mathcal{E}, \mathcal{V})_{y_0}$ is not of cotype $(b)$.  Repeated application of the Proposition \ref{subtypecotype} we conclude that $(\mathcal{E}, \mathcal{V})_{y_0}$ is not of subtype $(a)$ as we desired.
\end{enumerate}
\end{proof}

Now, we are ready to prove that subtype  is an open property for coherent systems.

\begin{theorem} \label{openproperty}
Let $\alpha >0$, $n \geq 1$, $k \geq 0$ integer numbers, $(a) \in \mathbb{Q}^{(n-1)(k+1)}$ and $Y$ be a variety.  If $(\mathcal{E}, \mathcal{V})$ is a family of coherent systems on $X$ of type $(n,d,k)$ parametrized by $Y$ then;
\begin{enumerate}
\item The set 
\[Y_n((b)) = \{y \in Y : (\mathcal{E}, \mathcal{V})_y \,\,\, \text{is not of cotype $(b)$}\}\]
is a closed set in $Y$ where $(b)$ is the sequence defined by
\[(b) := (b_{ij}= a_{n-i, k-j}\frac{n-i}{i} : 0 < i <n, 0 \leq j \leq k).\]
\item The set 
\[Y((a)) = \{y \in Y : (\mathcal{E}, \mathcal{V})_y \,\,\, \text{is  of subtype $(a)$}\}\]
is an open set in $Y$.
\end{enumerate}
\end{theorem}

\begin{proof}
\begin{enumerate}
\item Let $(a):=(a_{i,j})$, $(b) := (b_{ij}= a_{n-i, k-j}\frac{n-i}{i} : 0 < i <n, 0 \leq j \leq k) \in  \mathbb{Q}^{(n-1)(k+1)}$ and $(\mathcal{E}, \mathcal{V})$ be a family of coherent systems on $X$ of type $(n,d,k)$ parametrized by $Y$.  For any triple of integers $m'$, $d'$, $t'$, $0 < m' <n$, $0 \leq t' \leq k$ let us consider the set $\Delta(m',d',t')$ consisting  of isomorphism classes of coherent systems $(G,Z)$  such that
\begin{enumerate}
\item $(G,Z)$ is of type $(m',d',t')$,
\item $(\mathcal{E}, \mathcal{V})_y \to (G,Z) \to 0$ for some $y \in Y$,
\item $\mu_\alpha(G,Z) \leq \frac{d + \alpha k}{n} - b_{m',t'}$,
\item $(G,Z)$ is of cotype $(b - b_{m',t'})$.
\end{enumerate}

Let  $Quot_{(m',d',t')}(\mathcal{E}, \mathcal{V})$ denote the Quot-scheme which parametrizes all quotient coherent systems of $(\mathcal{E}, \mathcal{V})$ of type $(m',d',t')$ (see \cite[Section 1.6.]{MinHe}).  Denote by $p_{i,j}$ the canonical projection on $X \times Y \times Quot_{(m',d',t')}(\mathcal{E}, \mathcal{V})$ for $i,j = 0,1,2$. We denote by
\[p_{1,2}^*(\mathcal{E}, \mathcal{V}) \to (\mathcal{G}, \mathcal{Z}) \to 0\]
the universal quotient coherent system on $X \times Y \times Quot_{(m',d',t')}(\mathcal{E}, \mathcal{V})$. By Proposition \ref{subtypecotype}, Lemma \ref{lemmabounded} and property $(d)$ the set $\Delta(m',d',t')$ is bounded.  In particular, note that every member of $\Delta(m',d',t')$ is isomorphic to one $q \in Quot_{(m',d',t')}(\mathcal{E}, \mathcal{V})$.   Consider the sheaf 
\begin{equation}\label{sheaf}
\underline{Hom}_{p_{2,3}}(p_{1,2}^*(\mathcal{E}, \mathcal{V}), (\mathcal{G},\mathcal{Z}))
\end{equation}
on $Y \times Quot_{(m',d',t')}(\mathcal{E}, \mathcal{V})$ and denote by $\Gamma(m',d',t') $ the support of (\ref{sheaf}), that is
\[
\begin{aligned}
& \Gamma(m',d',t')  :=  Supp(\underline{Hom}_{p_{2,3}}(p_{1,2}^*(\mathcal{E}, \mathcal{V}), (\mathcal{G},\mathcal{Z})) \\
& = \{(y,q) \in Y \times Quot_{(m',d',t')}(\mathcal{E}, \mathcal{V}):  Hom ((\mathcal{E}, \mathcal{V}), (\mathcal{G},\mathcal{Z}))_{(y,q)} \neq 0\}.
\end{aligned}\]

We denote by $\pi_Y(\Gamma(m',d',t'))$ the image of $\Gamma(m',d',t')$ under the canonical projection $\pi_Y:Y \times Quot_{(m',d',t')}(\mathcal{E}, \mathcal{V}) \longrightarrow Y$.  Note that the set $\pi_Y(\Gamma(m',d',t'))$ is constructible.  Define the set $H$ by
\[H := \bigcup_{(m',d',t')}\pi_Y(\Gamma(m',d',t'))\]
which is constructible since it is a finite union of constructible sets.  We claim that 
\[H = Y_n((b)).\]
In fact, if $y \in Y_n((b))$ the Lemma \ref{nocotype} and the Proposition \ref{subtypecotype} implies that there exist integers $m'$, $t'$, 
$0 < m' <n$, $0 \leq t' \leq k$ and a quotient coherent system $(G,Z)$ of $(\mathcal{E}, \mathcal{V})_y$ such that $(G,Z)$ is of cotype $(b-b_{m',t'})$ and
\[\mu_\alpha(G,Z) \leq \frac{d+ \alpha k}{n} - b_{m',t'}.\]
It follows, that there exists $q \in Quot_{(m',d',t')}(\mathcal{E}, \mathcal{V})$ such that $(G,Z)$ is isomorphic to the element corresponding to $q$, then $(y,q) \in \Gamma(m',d',t')$.  Therefore $y \in \pi_Y(\Gamma(m',d',t') \subset H$ and hence $Y_n((b)) \subseteq H$. 

Conversely, if $y \in H$,  there exist integers $m'$, $d'$ and $t'$ and a point $q \in Quot_{(m',d',t')}(\mathcal{E}, \mathcal{V})$ such that $(y,q) \in \Gamma(m',d',t') \subset Y \times Quot_{(m',d',t')}(\mathcal{E}, \mathcal{V})$.  By  properties $(b)$ and $(c)$ we have $(\mathcal{E}, \mathcal{V})_y$ is not of cotype $(b)$. Hence $y \in Y_n((b))$ and $Y_n((b)) \subseteq H$. Therefore,  $H =Y_n((b))$.

Since $Y_n((b))$ is a constructible set and by Proposition \ref{specialization} it is stable under specialization, we conclude that $Y_n((b))$ is closed in $Y$ as we required.  

\item Let $(\mathcal{E}, \mathcal{V})$ be a family of coherent systems of type $(n,d,k)$ parametrized by $Y$.  By $(1)$ the set 
\[Y_n((b)) = \{y \in Y : (\mathcal{E}, \mathcal{V})_y \,\,\, \text{is not of cotype $(b)$}\}\]
is a closed set in $Y$, where $(b)$ is the sequence 
\[(b) := (b_{ij}= a_{n-i, k-j}\frac{n-i}{i} : 0 < i <n, 0 \leq j \leq k).\] 
By Proposition \ref{subtypecotype} the set $Y_n((b))$ is equivalent to the set 
\[Y^n((a)):= \{y \in Y : (\mathcal{E}, \mathcal{V})_y \,\,\, \text{is not of subtype $(a)$}\}.\]
Therefore,
\[Y((a)) = \{y \in Y : (\mathcal{E}, \mathcal{V})_y \,\,\, \text{is  of subtype $(a)$}\}\]
is an open set in $Y$ as we required.
\end{enumerate}
\end{proof}

Now we are able to prove that the $(m,t)-$Segre function is lower semicontinuous, (see Theorem \ref{Theosemicontinuous}).

\textbf{Proof of the Theorem \ref{Theosemicontinuous}.}   Let $(\mathcal{E}, \mathcal{V})$ be a family of coherent systems of type $(n,d,k)$ parametrized by $Y$. In order to prove that the $(m,t)-$Segre function 
\begin{align*}
S^\alpha_{m,t} : Y &\longrightarrow \mathbb{R} \cup \{\infty\} \\
y &\longmapsto S^\alpha_{m,t}(\mathcal{E}, \mathcal{V})_y 
\end{align*}
is lower semicontinuous we need to show that the set 
\[S(b):= \{y \in Y : S^\alpha_{m,t}(\mathcal{E}, \mathcal{V})_y > b\}\]
is open in $Y$ for any $b \in \mathbb{R}$.  Note that the set $S(b)$ is equivalent to the set
\[
S(b)  = \{y \in Y : \mu_\alpha(F,W) < \mu_\alpha(\mathcal{E}, \mathcal{V})_y- \frac{b}{nm}, \,\, \text{for any} \,\,(F,W) \in P_{m,t}(\mathcal{E}, \mathcal{V})_y \}. \]
Let $A(b)$ denote the set 
\[A(b) := \{(a):=(a_{i,j}) \in \mathbb{Q}^{(n-1)(k+1)}  : a_{m,t} = \frac{-b}{nm}, \,\, 0 < m <n, 0 \leq t \leq k\}.\]
From Theorem \ref{openproperty},  for any $(a) \in A(b)$ the set 
\[Y((a)) = \{y \in Y : (\mathcal{E}, \mathcal{V})_y \,\,\, \text{is  of subtype $(a)$}\}\]
is an open set in $Y$.  Note that  
\[\bigcup_{(a) \in A(b)} Y(a) = S(b).\]

Hence $S(b)$ is an open set in $Y$,  since it is an arbitrary union of open sets. Therefore, we conclude that Segre function $S^\alpha_{m,t}$ is lower semicontinuous as we required.

\section{Stratification of $G(\alpha;n,d,k)$ according to the invariant $S^\alpha_{m,t}$}

In this section we use the $(m,t)-$Segre invariant to induce a stratification of the moduli space $G(\alpha;n,d,k)$ of $\alpha-$stable coherent systems on $X$ of type $(n,d,k)$.   If $GCD(n,d,k)=1$ by \cite[Proposition A.8.]{Bradlow-Garcia-Prada-Mercat-Munoz-Newstead} there exists a universal family $(\mathcal{E}, \mathcal{V})$ parametrized by $G(\alpha;n,d,k)$.  If $GCD(n,d,k) \neq 1$ working locally in the \'etale topology we can assume that there is a family $(\mathcal{E}, \mathcal{V})$ parametrized by $G(\alpha;n,d,k)$.

Let $(\mathcal{E}, \mathcal{V})$ be a family of coherent systems on $X$ parametrized by $G(\alpha;n,d,k)$. From Theorem \ref{Theosemicontinuous} the $(m,t)-$Segre function is lower semicontinuous, hence it induces a strati-\\fication of the moduli space $G(\alpha;n,d,k)$  into locally closed subvarieties
\[G(\alpha;n,d,k;m,t;s) := \{(E,V) \in G(\alpha;n,d,k) : S^\alpha_{m,t}(E,V) = s\}\]
according to the value $s$ of $S^\alpha_{m,t}$.  Note that every pair of integers $m, t$, $0 < m < n$, $0 \leq t \leq k$ define a stratification of $G(\alpha;n,d,k)$.  That means we would have $(n-1)(k+1)$ different stratifications for the same moduli space $G(\alpha;n,d,k)$.  Moreover, if $\alpha_i$, $\alpha_{i+1}$ are consecutive critical values, then the stratification is independent of $\alpha$ within the interval $(\alpha_i, \alpha_{i+1})$.  It can be used to analyze the differences between the consecutive moduli spaces in the family $\{G_0, G_1, \ldots, G_L\}$.

One of the main difficulties in the study of the stratification of $G(\alpha;n,d,k)$ is to show that the different strata are non-empty.  Our strategy consists in to construct extensions of coherent systems 
\[0 \longrightarrow (F,W) \longrightarrow (E,V) \longrightarrow (G,Z) \longrightarrow 0\]
in which $(E,V)$ is $\alpha-$stable and $(F,W) \in P_{m,t}(E,V)$ is maximal.  In this section we determine values of $\alpha$, $m$, $t$ and $s$ under which the different strata are non-empty and determine their dimension.

Let $\alpha= \frac{p}{q} \in \mathbb{Q}^+$.  Suppose that the stratum
\begin{equation} \label{stratumimp}
G(\alpha;n,d,k;m,t;1/q)
\end{equation}
 is non-empty.  If $(E,V) \in G(\alpha;n,d,k;m,t;1/q)$, then there exists $(F,W) \in P_{m,t}(E,V)$ such that
\begin{equation}  \label{condition}
S^\alpha_{m,t}(E,V) = (mn) (\mu_\alpha(E,V)-\mu_\alpha(F,W))= \frac{1}{q},
\end{equation}
and an exact sequence
\[0 \to (F,W) \to (E,V) \to (G,Z) \to 0\]
where the quotient system $(G,Z)$ is of type $(m', d_G, t')$.   Note that (\ref{condition})  implies 
\begin{equation*} \label{conditionextension}
q(md-nd_F) + p (mk-nt) = 1
\end{equation*}
that is equivalent to 
\[q(md_G- m'd_F) + p (mt'- m't) = 1.\]

\begin{remark} \label{stabilitycondition} \emph{
Let $\alpha= \frac{p}{q} \in \mathbb{Q}^+$. If $(E,V)$ is $\alpha-$estable, then for any coherent subsystem $(F,W) \subset (E,V)$ of type $(m, d_F, t)$ we have 
\[0 < \frac{1}{qnm} \leq \mu_\alpha(E,V) - \mu_\alpha(F,W).\]}
\end{remark}

The following theorem establishes conditions under which the stratum (\ref{stratumimp}) is non-empty.

\begin{theorem} \label{Theoremnonempty}
Let $\alpha= \frac{p}{q} \in \mathbb{Q}^+$ and $n \geq 2$, $d > 0$, $k \geq 1$ be integer numbers.   Suppose that there exist $n_1, n_2, d_1, d_2 >0$, $t_1, t_2  \geq 0$ integer numbers such that
\begin{equation} \label{existextension}
n_1n_2(g-1) -d_1n_2 + d_2n_1 + t_2n_1(g-1) - t_1t_2 > 0,
\end{equation}
\begin{equation} \label{princconditiontheorem}
q(n_1d_2-n_2d_1)+p(n_1t_2-n_2t_1) = 1.
\end{equation}
and  the moduli spaces $G(\alpha;n_1,d_1,t_1)$ and $G(\alpha;n_2,d_2,t_2)$ are non-empty.  Then, the stratum 
\[G(\alpha;n,d,k;n_1,t_1;1/q)\]
is non-empty, where $n= n_1+n_2$, $d= d_1+d_2$ and $k= t_1+t_2$.
\end{theorem}

The proof of the theorem makes use of the following results.

\begin{theorem} \label{Theorem1}
Let $\alpha= \frac{p}{q} \in \mathbb{Q}^+$ and $(E,V) \in G(\alpha;n,d,k)$.  If  $(F,W) \in P_{m,t}(E,V)$ and
\begin{equation} \label{conditiontheorem}
q(md - nd_F) + p(mk-nt) =1
\end{equation} 
then, $(F,W)$ is $\alpha-$stable and maximal.  Moreover, the quotient system $(G,Z)$ defined by $(F,W)$ is $\alpha-$stable.
\end{theorem}

\begin{proof}
Let $\alpha = \frac{p}{q} \in \mathbb{Q}^+$, $(E,V) \in G(\alpha;n,d,k)$ and $(F,W) \in P_{m,t}(E,V)$ which satisfy $(\ref{conditiontheorem})$.  Let $(G,Z)$ the quotient coherent system of $(E,V)$ of type $(m',d_G,t')$ that fit in the exact sequence, 
\[0 \longrightarrow (F,W) \longrightarrow (E,V) \longrightarrow (G,Z) \longrightarrow 0.\]
 We claim that  $(F,W)$ and $(G,Z)$ are $\alpha-$stable coherent systems.
 
\begin{enumerate}
\item Let $(F_1,W_1)$ be a subsystem of $(F,W)$ of type $(n_1,d_{F_1},k_1)$, $n_1 \leq m$.
If $n_1= m$ we would have $F \cong F_1$ and $dim \,\, W_1 < dim \,\, W$, hence $\mu_\alpha(F_1,W_1) < \mu_\alpha(F,W)$.  If $n_1 < m$ since $(E,V)$ is $\alpha-$stable by Remark \ref{stabilitycondition} we have
\begin{equation} \label{inequlityproof}
\mu_\alpha(F_1,W_1) \leq \mu_\alpha(E,V) - \frac{1}{qnn_1}.
\end{equation}
and by $(\ref{conditiontheorem})$,
\begin{equation} \label{conditiontheoremimplies}
\mu_\alpha(E,V) =  \mu_\alpha(F,W) + \frac{1}{qnm}.
\end{equation}
Replacing (\ref{conditiontheoremimplies}) in $(\ref{inequlityproof})$ we obtain
\begin{align*}
\mu_\alpha(F_1,W_1) & \leq \mu_\alpha(F,W) + \frac{1}{qnm} - \frac{1}{qnn_1} \\
&= \mu_\alpha(F,W) - \frac{1}{qn} \left(\frac{m-n_1}{mn_1}\right), \,\, 0 < m-n_1 \\
& < \mu_\alpha(F,W),
\end{align*}
Hence, $(F,W)$ is $\alpha-$stable as we required.
\item In order to prove that $(G,Z)$ is $\alpha-$stable, note by \cite[Section 1.2.]{MinHe} that for any coherent system $(G_1,Z_1)$ the functor $Hom((G_1,Z_1),-)$ is left exact. Applying this functor to the extension 
\[0 \to (F,W) \to (E,V) \to (G,Z) \to 0\]
we have the induced homomorphism
\begin{equation} \label{homomorphisminduced}
Hom((G_1,Z_1),(G,Z)) \longrightarrow Ext^1((G_1,Z_1), (F,W)).
\end{equation}
Suppose that $(G_1,Z_1)$ is a subsystem of $(G,Z)$ of type $(n_1', d_1', k_1')$ and consider the exact sequence
\[0 \to (F,W) \to (E_1,V_1) \to (G_1,Z_1) \to 0\]
which by $(\ref{homomorphisminduced})$ defines the following diagram
\[
\begin{diagram}
\node{0} \arrow{e} \node{(F,W)} \arrow{e} \arrow{s,t}{=} \node{(E_1,V_1)} \arrow{e} \arrow{s} \node{(G_1,Z_1)} \arrow{e} \arrow{s} \node{0} \\
\node{0} \arrow{e} \node{(F, W)} \arrow{e}  \node{(E, V)} \arrow{e}  \node{(G, Z)} \arrow{e}  \node{0} \\
\end{diagram}\]
where  $(E_1, V_1)$ is of type  $(n',d_{E_1},k')$ and
\begin{equation} \label{mualphaextension}
\mu_\alpha(G_1,Z_1) = \mu_\alpha(E_1,V_1)\frac{n'}{n_1'}-\mu_\alpha(F,W)\frac{m}{n_1'}. 
\end{equation}
Since $(E,V)$ is $\alpha-$stable by Remark \ref{stabilitycondition} we have
\begin{equation} \label{impliesremark1}
\mu_\alpha(E_1,V_1) \leq \mu_\alpha(E,V) - \frac{1}{qnn'}
\end{equation}
and
\begin{equation}\label{impliesremark2}
\mu_\alpha(F,W) \leq \mu_\alpha(E,V)- \frac{1}{qnm}
\end{equation}
Replacing (\ref{impliesremark1}) and (\ref{impliesremark2}) in (\ref{mualphaextension}) we obtain
\begin{align*}
\mu_\alpha(G_1,Z_1) & \leq \left(\mu_\alpha(E,V)- \frac{1}{qnn'}\right)\frac{n'}{n'_1} - \left(\mu_\alpha(E,V)- \frac{1}{qnm}\right)\frac{m}{n'_1}, \,\, n'+m=n'_1 \\
& = \mu_\alpha(E,V).
\end{align*}
As $(E,V)$ is $\alpha-$stable we conclude that
\[\mu_\alpha(G_1,Z_1) \leq \mu_\alpha(E,V) < \mu_\alpha(G,Z).\]
Therefore $(G,Z)$ is $\alpha-$stable as we required.
\end{enumerate}
Finally, we prove that $(F,W) \in P_{m,t}(E,V)$ is maximal, suppose $(F,W)$ is not maximal then there exists $(F',W') \in P_{m,t}(E,V)$ such that $\mu_\alpha(F,W) < \mu_\alpha(F',W')$. Thus
\begin{align*}
0  &<  qnm (\mu_\alpha(E,V)- \mu_\alpha(F',W')) < qnm (\mu_\alpha(E,V)-\mu_\alpha(F,W)) \\
&= q(md-nd_F)+ p(mk-nt) = 1
\end{align*}
which is a contradiction, because $qnm(\mu_\alpha(E,V)- \mu_\alpha(F',W'))$ is an integer number.  Hence,  $(F,W) \in P_{m,t}(E,V)$ is maximal which completes the proof.
\end{proof}

\begin{theorem}\label{Theorem2}
Let $\alpha= \frac{p}{q} \in \mathbb{Q}^+$ and $(F,W)$, $(G,Z)$ be $\alpha-$stable coherent systems of type $(m,d_F, t)$ and $(m',d_G,t')$, respectively which satisfies
\begin{equation}\label{conditionTheorem2}
q(md_G-m'd_F)+ p (mt'-m't) = 1.
\end{equation}
If 
\[0 \to (F,W) \to (E,V) \to (G,Z) \to 0\]
is a non-trivial extension of coherent systems, then the coherent system $(E,V)$ is $\alpha-$stable and
\[S^\alpha_{m,t}(E,V) = (nm)(\mu_\alpha(E,V)- \mu_\alpha(F,W))= \frac{1}{q}.\]
\end{theorem}

\begin{proof}
Let $\alpha = \frac{p}{q} \in \mathbb{Q}^+$ and $(F,W)$, $(G,Z)$ be $\alpha-$stable coherent systems of type $(m,d_F,t)$, $(m',d_G,t')$, respectively which satisfies $(\ref{conditionTheorem2}$).  Suppose that there exists a non-trivial extension
\begin{equation} \label{extmaximal}
0 \to (F,W) \to (E,V) \to (G,Z) \to 0
\end{equation}
where  $(E,V)$ is of type $(n,d,k)$.  Let $(E',V')$ a subsystem of $(E,V)$ of type $(n',d',k')$ and  $(G_1,Z_1)$, $(F_1,W_1)$ be  the image and the kernel of the morphism $(E',V') \to (G,Z)$, respectively.
We have the following diagram
\[\begin{diagram}
\node{0} \arrow{e} \node{(F_1, W_1)} \arrow{e} \arrow{s} \node{(E', V')} \arrow{e} \arrow{s} \node{(G_1, Z_1)} \arrow{e} \arrow{s} \node{0} \\
\node{0} \arrow{e} \node{(F, W)} \arrow{e} \node{(E, V)} \arrow{e}  \node{(G,Z)} \arrow{e} \node{0} 
\end{diagram}\]
where $(F_1,W_1)$ and $(G_1,Z_1)$ are coherent systems of type $(n_1, d_{F_1}, t_1)$ and $(n'_1, d_{G_1}, t'_1)$, respectively.
Since $(F,W)$ and $(G,Z)$ are $\alpha-$stable by Remark \ref{stabilitycondition}, we have
\begin{equation} \label{inequality1}
\mu_\alpha(F_1,W_1) \leq \mu_\alpha(F,W) - \frac{1}{qmn_1} 
\end{equation}
and
\begin{equation} \label{inequality2}
\mu_\alpha(G_1,Z_1) \leq \mu_\alpha(G,Z) - \frac{1}{qm'n'_1} . 
\end{equation}
Also, by (\ref{conditionTheorem2}) we have
\begin{equation} \label{inequality31}
\mu_\alpha(F,W) = \mu_\alpha(E,V) - \frac{1}{qnm}.
\end{equation}
and
\begin{equation} \label{inequality3}
\mu_\alpha(G,Z) = \mu_\alpha(E,V) + \frac{1}{qnm'}.
\end{equation}

To show that $(E,V)$ is $\alpha-$stable we consider all possibilities of  $(F_1,W_1)$ and $(G_1,Z_1)$.  We claim $\mu_\alpha(E',V') < \mu_\alpha(E,V)$;
\begin{enumerate}[label=(\roman*)]
\item Assume $(F_1, W_1) = 0$.  Hence $(E',V') \cong (G_1,Z_1)$.  From (\ref{inequality2}) and (\ref{inequality3}) follows that
\begin{align*}
\mu_\alpha(E', V') & = \mu_\alpha(G_1, Z_1)   \\
                             & \leq \mu_\alpha(G,Z) - \frac{1}{qm'n'_1}\\
                             & = \mu_\alpha(E,V) - \frac{1}{qm'}\left(\frac{1}{n'_1}- \frac{1}{n}\right), \,\,    0  < n- n_1' \\
                             & < \mu_\alpha(E,V).
\end{align*}

\item Assume $(G_1, Z_1) =0$. We have $(E', V') \cong (F_1, W_1)$.  From (\ref{inequality31}) and since $(F,W)$ is $\alpha-$stable follows that
\[\mu_\alpha(E', V') = \mu_\alpha(F_1,W_1) <  \mu_\alpha(F, W) < \mu_\alpha(E,V).\]

\item  Assume $(G_1, Z_1) = (G, Z)$. Hence $(F_1,W_1) \neq (F, W)$ and 
\begin{equation} \label{desigualdad2}
\mu_\alpha(E',V')  = \mu_\alpha(F_1,W_1) \frac{n_1}{n'} + \mu_\alpha(G,Z) \frac{m'}{n'}.
\end{equation}
Replacing (\ref{inequality1}) in (\ref{desigualdad2}) we obtain
\begin{equation}\label{equ}
\mu_\alpha(E',V') \leq \left(\mu_\alpha(F,W)- \frac{1}{qmn_1}\right)\frac{n_1}{n'} + \mu_\alpha(G,Z)\frac{m'}{n'}.
\end{equation}
Replacing (\ref{inequality31}) and (\ref{inequality3}) in  (\ref{equ})  follows that
\begin{align*}
\mu_\alpha(E',V') & \leq \left(\mu_\alpha(E,V)-\frac{1}{qmn}- \frac{1}{qmn_1}\right) \frac{n_1}{n'} + \left(\mu_\alpha(E,V)+ \frac{1}{qm'n}  \right) \frac{m'}{n'}  \\
                             & = \mu_\alpha(E,V) - \frac{n_1}{qmnn'}-\frac{1}{qn'}\left(\frac{n-m}{nm}\right) ,\,\,\,\,\,\,\,\, 0 < n- m\\
                             & < \mu_\alpha(E,V).
\end{align*}

\item  Assume $(F_1, W_1) = (F, W)$.  Hence $(G_1,Z_1) \neq (G,Z)$ and 
\begin{equation*} 
\mu_\alpha(E',V')  = \mu_\alpha(F,W) \frac{m}{n'} + \mu_\alpha(G_1,Z_1) \frac{n_1'}{n'}.
\end{equation*}  
We proceed as in (iii) and  conclude that $\mu_\alpha(E', V') < \mu_\alpha(E,V)$.

\item   Assume $(F_1, W_1) \neq (F, W)$ and $(G_1,Z_1) \neq (G,Z)$. It follows that 
\begin{equation}\label{des4}
 \mu_\alpha(E', V')  = \mu_\alpha(F, W_1) \frac{n_1}{n'} + \mu_\alpha(G_1, Z_1) \frac{n'_1}{n'}.
\end{equation}
Replacing (\ref{inequality1}) and (\ref{inequality2}) in (\ref{des4}) we have
\begin{equation} \label{des7}
 \mu_\alpha(E', V') \leq \left( \mu_\alpha(F, W) - \frac{1}{qmn_1} \right) \frac{n_1}{n'} + \left( \mu_\alpha(G,Z) - \frac{1}{qm'n'_1} \right) \frac{n'_1}{n'}.
\end{equation}
Replacing (\ref{inequality31}) and (\ref{inequality3}) in (\ref{des7}), it follows that
\begin{align*}
 \mu_\alpha(E', V')  & \leq  \left( \mu_\alpha(E,V) - \frac{1}{qmn}- \frac{1}{qmn_1} \right) \frac{n_1}{n'} + \left(  \mu_\alpha(E,V) + \frac{1}{qm'n}- \frac{1}{qm'n'_1} \right) \frac{n'_1}{n'} \\
                              & = \mu_\alpha(E,V) - \frac{n_1}{qmn'}\left(\frac{1}{n}+ \frac{1}{n_1} \right) - \frac{n'_1}{qm'n'} \left(\frac{n-n'_1}{nn'_1}\right), \,\, 0 < n-n'_1 \\
                              & < \mu_\alpha(E,V). 
\end{align*}
\end{enumerate}
From (i), (ii), (iii), (iv) and (v), it follows that $\mu_\alpha(E',V') < \mu_\alpha(E,V)$. Therefore, $(E,V)$ is $\alpha-$stable as we required.  

Since $(E,V)$ is $\alpha-$stable and $(F,W)$ fit into the sequence (\ref{extmaximal}) by Theorem \ref{Theorem1}, the system $(F,W) \in P_{m,t}(E,V)$ is maximal and
\[S^\alpha_{m,t}(E,V) = (mn) (\mu_\alpha(E,V)- \mu_\alpha(F,W))\]
which completes the proof.
\end{proof}

\textbf{Proof of the Theorem \ref{Theoremnonempty}.} Let $\alpha= \frac{p}{q}$ and $n_1, n_2, d_1, d_2, t_1, t_2$ be integer numbers that satisfy the hypothesis of the theorem.  Let $(F,W) \in G(\alpha;n_1,d_1,t_1)$ and $(G,Z) \in G(\alpha;n_2,d_2,t_2)$.  From (\ref{existextension}) and \cite[Proposition 3.2.]{Bradlow-Garcia-Prada-Munoz-Newstead1}, 
\[dim \,\, Ext^1((G,Z),(F,W)) > 0,\]
hence there exists a non-trivial extension 
\[0 \to (F,W) \to (E,V) \to (G,Z) \to 0.\]
As $(F,W)$, $(G,Z)$ are $\alpha-$stable and these satisfy (\ref{princconditiontheorem}) from Theorems \ref{Theorem1} and  \ref{Theorem2} we conclude  that  $(E,V)$ is $\alpha-$stable and  $(F,W) \in P_{m,t}(E,V)$ is maximal. Therefore $G(\alpha;n,d,k;m,t;1/q)$ is non-empty which completes the proof.

The following theorem gives us the dimension of the stratum $G(\alpha;n,d,k;m,t;1/q)$.

\begin{theorem} \label{dimension}
Under the hypothesis of Theorem \ref{Theoremnonempty}.  Suppose that  
\[dim \,\, Ext^2((F_2,W_2),(F_1,W_1)) = cte\]
for any $(F_i,W_i) \in G(\alpha;n_i, d_i, t_i)$ for $i= 1,2$.  Then, the dimension of the stratum $G(\alpha;n,d,k;n_1,t_1;1/q)$ is bounded  above by
\[dim \,\, G(\alpha;n_1,d_1,t_1) + dim \,\, G(\alpha;n_2,d_2,t_2) + C_{21}-1,\]
where $C_{21}= n_1n_2(g-1)-d_1n_2+d_2n_1 +t_2d_1-t_2n_1(g-1)-t_1t_2.$
\end{theorem}

\begin{proof}
Let $\alpha = \frac{p}{q} \in \mathbb{Q}^+$ and  $n_1, n_2, d_1, d_2 > 0$, $t_1, t_2 \geq 0$ be integer numbers that satisfy the hypothesis of the Theorem \ref{Theoremnonempty}.  Working locally in the \'etale topology if necessary, we can assume without loss of generality that there exists a family  $(\mathcal{F}_i, \mathcal{W}_i)$ of coherent systems on $X$ of type $(n_i,d_i,t_i)$ parametrized by $G_i:=G(\alpha;n_i,d_i,t_i)$, $i=1,2$.  Let $p_{i,j}$ denote the canonical projections of $X \times G_1 \times G_2$ for $i,j = 0,1,2$.  By Theorem \ref{Theorem1} follows that in the exact sequence 
\begin{equation} \label{sequence}
0 \longrightarrow (F_1,W_1) \longrightarrow (E,V) \longrightarrow (F_2,W_2) \longrightarrow 0
\end{equation}
$(E,V)$ is $\alpha-$stable, it implies that $dim \,\, Hom((F_2,W_2), (F_1,W_1))=0$ (see \cite[Corollary 2.5.1]{King-Newstead}).  Suppose that $dim \,\, Ext^2((F_2,W_2),(F_1,W_1)) = cte$, then by \cite[Proposition 3.2.]{Bradlow-Garcia-Prada-Munoz-Newstead1} we have 
\[dim \,\, Ext^1((F_2,W_2),(F_1,W_1)) = C_{21}+cte.\]
By \cite[Corollaire 1.20]{MinHe}, there is a vector bundle 
\[\Gamma := \underline{Ext}^1(p_{0,2}^*(\mathcal{F}_2, \mathcal{W}_2), p_{0,1}^*(\mathcal{F}_1, \mathcal{W}_1))\]
over $G_1 \times G_2$ whose fibre over $((F_1,W_1),(F_2,W_2))$ is $Ext^1((F_2,W_2), (F_1,W_1))$ for all $(F_i,W_i)$, $i= 1,2$. Note that $\mathbb{P}\Gamma$ parametrizes the non-trivial extensions (\ref{sequence}) up to scalar multiples.  

By \cite[Lemma A.10]{Bradlow-Garcia-Prada-Mercat-Munoz-Newstead}, there exists a universal extension 
\begin{equation} \label{universal}
0 \longrightarrow (id \times \pi)^*p_{0,1}^*(\mathcal{F}_1, \mathcal{W}_1) \otimes \mathcal{O}_{\mathbb{P}\Gamma}(1) \longrightarrow (id \times \pi)^*(\mathcal{E}, \mathcal{V}) \longrightarrow (id \times \pi)^*p_{0,2}^*(\mathcal{F}_2, \mathcal{W}_2) \longrightarrow 0
\end{equation}
on $X \times \mathbb{P}\Gamma$.

Define the set
\[U := \{p \in \mathbb{P}\Gamma: (\mathcal{E}, \mathcal{V})_p \,\, \text{is $\alpha-$stable and $S^\alpha_{m,t} (\mathcal{E}, \mathcal{V})_p = 1/q$}\}.\]

From the lower semicontinuous of the function $S^\alpha_{m,t}$, Theorem \ref{Theoremnonempty} and $\alpha-$stability being an open condition we conclude that the set $U$ is non-empty and open in $\mathbb{P}\Gamma$.   Restricting the sequence (\ref{universal}) on $X \times U$ from the universal property of the moduli space $G(\alpha;n,d,k)$ we have a morphism
\[f: U \longrightarrow G(\alpha;n,d,k;n_1,t_1;1/q) \subset G(\alpha;n,d,k).\]
Note that if $p \in U$, then $f(p)$ is precisely the point of $G(\alpha;n,d,k)$ representing to $(E,V)$.  We now determine the dimension of the stratum
\begin{align*}
 dim \,\, & G (\alpha;n,d,k;m,t;1/q)  =  dim \,\, U - dim \,\, f^{-1}(E,V) \\
                                   & =  dim \,\, \mathbb{P}\Gamma - dim \,\, f^{-1}(E,V) \\
                                   & =  dim \,\, G_1 + dim \,\,G_2 + C_{21} -1  - dim \,\, f^{-1}(E,V)\\ 
                                   & \leq  dim \,\, G_1 + dim \,\,G_2  +  C_{21} -1,  
\end{align*}
where $(F_i, W_i) \in G_i$ for $i=1,2$.  This proves the theorem.
\end{proof}

In general, it is not easy to compute the dimension of $ Ext^2((F_2,W_2),(F_1,W_1))$. Howe-\\ver, to establish a better bound of the dimension of the stratum $G(\alpha;n,d,k;n_1,t_1;1/q)$ in the Theorem \ref{dimension} we could define a stratification $\{S_t\}$ of $ G(\alpha;n_1,d_1,t_1) \times G(\alpha;n_2,d_2,t_2)$ such that on each $S_t$ we have that
\begin{equation} \label{ext2}
dim \,\ Ext^2((F_2,W_2),(F_1,W_1)) := a.
\end{equation} 
Hence
\begin{equation*}
dim \,\ Ext^1((F_2,W_2),(F_1,W_1)) := C_{2,1}+a
\end{equation*}
will be constant on each $S_t$.  By \cite[Lemma 3.3.]{Bradlow-Garcia-Prada-Munoz-Newstead1} the quantity (\ref{ext2}) is bounded on $G(\alpha;n_1,d_1,t_1) \times G(\alpha;n_2,d_2,t_2)$.  Taking the maximum of these dimensions we have a better bound of $dim \,\, G(\alpha;n,d,k;n_1,t_1;1/q)$. 

\section{Applications to cross critical values}

In this section we apply the previous result for a particular case.  Let $X$ be a general curve of genus $6$ and $(n,d,k) = (2,13,4)$.  For  $(2,13,4)$ the non-zero virtual critical values belong to 
\[\{\frac{2d'-13}{4-2k'} : 0 \leq k' \leq 4\} \cap (0 , \infty).\]
Here $1/4$ and $1/2$ are the first virtual critical values.  Note that $\alpha=1/4$ is not a critical value, because there is no a coherent system $(E,V)$ and a subsystem $(L_1,W_1)$ of type $(1,6,4)$ on a general curve of genus $6$ (see \cite{Arbarello-Cornalba-Grffiths-Harris}).  The first critical value is $\alpha= 1/2$ since the coherent system $(E,V) := (L_1 \oplus L_2, W_1 \oplus W_2)$ satisfies $\mu_\alpha(E,V) = \mu_\alpha(L_1,W_1)$ where  $(L_1,W_1)$ and $(L_2,W_2)$ are coherent systems of type $(1,6,3)$ and $(1,7,1)$, respectively.  Denote by $\alpha_0 = 0$, $\alpha_1= 1/2$ the first critical values.

Note that for any $p \in \mathbb{Z}^+$, $\alpha_p := \frac{p}{2p+1}$ is in the interval $(\alpha_0, \alpha_1)$, hence $G_0(2,13,4)= G(\alpha_p;2,13,4)$.  Since $n=2$ and $k=4$, we have $5$ different stratifications for the moduli space $G_0(2,13,4)$ each one induced by the functions $S^{\alpha_{p}}_{1,0}, S^{\alpha_{p}}_{1,1}, S^{\alpha_{p}}_{1,2}, S^{\alpha_{p}}_{1,3}$ and $S^{\alpha_{p}}_{1,4}$, respectively.

For instance, for $S^{\alpha_{p}}_{1,3}: G_0(2,13,4) \longrightarrow \mathbb{R} \cup \{\infty\}$, we have that $S^{\alpha_p}_{1,3}(E,V) \leq 10-2\alpha_p$ provided that $P_{1,3}(E,V) \neq \emptyset$ .  Moreover, from $\alpha-$stability and by classical Brill-Noether theory (see \cite{Arbarello-Cornalba-Grffiths-Harris}), it follows that 
\[S^{\alpha_{p}}_{1,3}(E,V)= \begin{cases}
s_1:= \frac{1}{2p+1}, \,\, &\text{if $P_{1,3}(E,V) \neq \emptyset$} \\
\infty, \,\, &\text{if $P_{1,3}(E,V) = \emptyset$.}
\end{cases}\]

Hence the function $S^{\alpha_{p}}_{1,3}$ induces the stratification \[G_0(2,13,4) := G_0(2,13,4;1,3;s_1) \sqcup G_0(2,13,4;1,3;\infty).\]

From Theorem \ref{Theoremnonempty} it follows that  the stratum $G_0(2,13,4;1,3;s_1)$ is non-empty.  Moreover by Theorem \ref{dimension} and \cite[Lemma 3.3.]{Bradlow-Garcia-Prada-Munoz-Newstead1},  we get $dim \,\, G_0(2,13,4;1,3;s_1) =11$.  Since every irreducible component $G$ of $G_0(2,13,4)$ has dimension $G \geq \beta(2,13,4) = 17$  (see \cite[Corollaire 3.14.]{MinHe}) we conclude that the stratum $G_0(2,13,4;1,3;\infty)$ is non-empty. 

This stratification yields information about of how to change the moduli space $G_0(2,13,4)$ when it crosses the critical value $\alpha_1$.  

\begin{theorem}
Let $X$ be a general curve of genus $6$ and $G_0(2,13,4)$ be the moduli space in the interval $(\alpha_0, \alpha_1)$.  
Then, the stratum $G_0(2,13,4;1,3;s_1)$ induced by $S^{\alpha_{p}}_{1,3}$ is not a subset of $G_1(2,13,4)$.
\end{theorem} 

\begin{proof}
Note that any $(E,V) \in G_0(2,13,4;1,3;s_1) \subset G_0(2,13,4)$ can be written in an exact sequence 
\[0 \longrightarrow (L_1,W_1) \longrightarrow (E,V) \longrightarrow (L_2,W_2) \longrightarrow 0\]
where $(L_1,W_1)$ and $(L_2,W_2)$ are coherent systems of type $(n_1,d_1,t_1) =(1,6,3)$ and $(n_2,d_2 ,t_2) =(1,7,2)$, respectively and  \[S^{\alpha_{p}}_{1,3}(E,V) = 2 (\mu_{\alpha_{p}}(E,V)-\mu_{\alpha_{p}}(L_1,W_1)) = s_1.\]  For the critical value $\alpha_1= 1/2$, we get
\[\mu_{\alpha_{1}}(E,V) = \frac{13 + 2 \alpha_1}{2} = \mu_{\alpha_{1}}(L_1,W_1) = 6 + 3\alpha_1.\]
Therefore by \cite[Lemma 6.2.]{Bradlow-Garcia-Prada-Munoz-Newstead1}, $(E,V)$ is unstable for all $\alpha > \alpha_1$.  Hence \\$G(\alpha;2,13,4;1,3;s_1) \nsubseteq G_1(2,13,4)$ which proves the theorem.
\end{proof}

\def\refname{R\MakeLowercase{eferences}}

\end{document}